\documentclass[a4paper,reqno]{amsart}
\usepackage{amssymb, amsmath, amscd}
\usepackage[final]{graphicx}
\usepackage{wrapfig}
\usepackage{color}

\newtheorem{theorem}{Theorem}[section]
\newtheorem{lemma}[theorem]{Lemma}

\newtheorem{definition}[theorem]{Definition}

\makeatletter
 \@addtoreset{equation}{section}
\makeatother\begin{document}
\title[Moduli space]{The moduli space of Type~$\mathcal{A}$ surfaces with torsion
and non-singular symmetric Ricci tensor}
\author{Peter B Gilkey}
\address{Mathematics Department, University of Oregon, Eugene OR 97403 USA}
\email{gilkey@uoregon.edu}
\begin{abstract} We examine the moduli spaces of Type~$\mathcal{A}$ connections on oriented and
unoriented surfaces both with
and without torsion in relation to the signature of the associated symmetric Ricci tensor $\rho_s$.
If the signature of $\rho_s$ is $(1,1)$ or $(0,2)$,
the spaces are smooth. If the signature is $(2,0)$, there is an orbifold singularity.
\end{abstract}
\maketitle
\section{Introduction}
Let $\nabla$ be a connection on the tangent bundle of a smooth surface $M$. Introduce a system of local
coordinates $x:=(x^1,x^2)$ on $M$ to expand $\nabla_{\partial_{x^i}}\partial_{x^j}=\Gamma_{ij}{}^k\partial_{x^k}$ where
we adopt the {\it Einstein convention} and sum over repeated indices. 
The components of the curvature operator $R$ are given by:
$$
R_{ijk}{}^l=\partial_{x_i}\Gamma_{jk}{}^l-\partial_{x^j}\Gamma_{ik}{}^l
+\Gamma_{in}{}^l\Gamma_{jk}{}^n-\Gamma_{jn}{}^l\Gamma_{ik}{}^n\,.
$$
We say that $\mathcal{M}$ is {\it symmetric} if $\nabla R=0$.
Let $\rho(x,y):=\operatorname{Tr}\{z\rightarrow R(z,x)y\}$ be the {\it Ricci tensor} and
let $\nabla\rho$ be the covariant derivative of the Ricci tensor. Although the Ricci tensor is a symmetric 2-tensor
in the Riemannian setting, it need not be symmetric in this more general setting and we define the {\it symmetric Ricci tensor} by setting
$\rho_s(X,Y):=\frac12\{\rho(X,Y)+\rho(Y,X)\}$. We say that $\mathcal{M}:=(M,\nabla)$ is 
{\it locally homogeneous} if given any two points $P$ and $Q$ of $M$, there is the germ of a local diffeomorphism
$\Phi_{P,Q}$ taking $P$ to $Q$ which commutes with with $\nabla$. 

We say that $\nabla$ is {\it torsion free} if
$\nabla_XY-\nabla_YX-[X,Y]=0$; this symmetry is equivalent to the symmetry $\Gamma_{ij}{}^k=\Gamma_{ji}{}^k$.
Torsion free connections on surfaces have been used to construct new examples of pseudo-Riemannian metrics exhibiting properties without Riemannian counterpart \cite{CGGV09, CGV10, De, KoSe}.

The theory of connections with torsion plays an important role in string theory \cite{A06,FI02,GMW04,I04},
they are important in almost contact geometry \cite{FI03,IM14,M12,S12}, they play a role in non-integrable
geometries \cite{A06,ACF05,ACFH15,B15}, they are important in spin geometries \cite{K10}, they are useful in
considering almost hypercomplex geometries \cite{M11}, they appear in the study of
compact solvmanifolds \cite{F15}, and they have been used to study the non-commutative residue
for manifolds with boundary \cite{WWY14}. The following result was first proved in the torsion free
setting by  Opozda \cite{Op04} and subsequently extended to 
surfaces with torsion by Arias-Marco and Kowalski \cite{AMK08}, see also 
 \cite{AMK08, Du, G-SG, KVOp2, KVOp, Opozda} for related work.
\begin{theorem}
Let $\mathcal{M}=(M,\nabla)$ be a locally homogeneous surface where $\nabla$
can have torsion. Then at least one of the following
three possibilities, which are not exclusive, hold which describe the local geometry:
\begin{itemize}
\item[($\mathcal{A}$)] There exists a coordinate atlas so the Christoffel symbols
$\Gamma_{ij}{}^k$ are constant.
\item[($\mathcal{B}$)] There exists a coordinate atlas so the Christoffel symbols have the form
$\Gamma_{ij}{}^k=(x^1)^{-1}C_{ij}{}^k$ for $C_{ij}{}^k$ constant and $x^1>0$.
\item[($\mathcal{C}$)] $\nabla$ is the Levi-Civita connection of a metric of constant Gauss
curvature.
\end{itemize}\end{theorem}

Such a surface which is not flat is said to be of Type-$\mathcal{A}$, Type-$\mathcal{B}$, or Type-$\mathcal{C}$ depending upon which
of the possibilities hold in this result. These classes are not disjoint. While
there are no surfaces which are both Type-$\mathcal{A}$ and Type-$\mathcal{C}$,
there are surfaces which are both Type-$\mathcal{A}$ and Type-$\mathcal{B}$
and there are surfaces which are both Type-$\mathcal{B}$ and Type-$\mathcal{C}$.
We refer to the discussion in \cite{BGGP16} for further details.

We shall work in the Type~$\mathcal{A}$ setting for the remainder of this paper and
shall let $\mathcal{M}:=(M,\nabla)$ be a locally homogeneous Type~$\mathcal{A}$
surface with torsion. In this setting, the Christoffel symbols
$$
\Gamma=\{\Gamma_{11}{}^1,\Gamma_{11}{}^2,\Gamma_{12}{}^1,\Gamma_{21}{}^1,
\Gamma_{12}{}^2,\Gamma_{21}{}^2,\Gamma_{22}{}^1,\Gamma_{22}{}^2\}
$$
belong to $\mathbb{R}^8$. For $p+q=2$, let $\mathcal{W}(p,q)\subset\mathbb{R}^8$ be the open subset of Christoffel symbols
defining a Type~$\mathcal{A}$ structure such that the associated symmetric Ricci tensor $\rho_{s,\Gamma}$
is non-degenerate and of signature $(p,q)$;
in contrast to the torsion free setting, the Ricci tensor can be non-symmetric and thus it is necessary to deal with the symmetrization
which was defined previously.
Let $\mathfrak{W}^+(p,q)$ (resp. $\mathfrak{W}(p,q)$) be the corresponding moduli space of oriented (resp. unoriented)
Type~$\mathcal{A}$ structures.

The general linear group $\operatorname{GL}(2,\mathbb{R})$ acts on $\mathcal{W}(p,q)$ and on $\mathcal{Z}(p,q)$
in the obvious fashion by changing the basis for $\mathbb{R}^2$; we shall denote this action by $g\Gamma$. Let 
$\operatorname{GL}^+(2,\mathbb{R})\subset\operatorname{GL}(2,\mathbb{R})$ be the connected
component of the identity; these are the matrices with positive determinant. The action of the general linear group completely
determines the isomorphism type in this setting.

\begin{theorem}
$\mathfrak{Z}^+(p,q)=\mathcal{Z}(p,q)/\operatorname{GL}^+(2,\mathbb{R})$,\hglue .45cm
$\mathfrak{W}^+(p,q)=\mathcal{W}(p,q)/\operatorname{GL}^+(2,\mathbb{R})$,
\medbreak\hglue 2.1cm
$\mathfrak{Z}(p,q)=\mathcal{Z}(p,q)/\operatorname{GL}(2,\mathbb{R})$,\hglue 1cm
$\mathfrak{W}(p,q)=\mathcal{W}(p,q)/\operatorname{GL}(2,\mathbb{R})$.
\end{theorem}

\begin{proof}
Let $\mathcal{M}=(M,\nabla)$ be a Type~$\mathcal{A}$ surface with non-degenerate symmetric Ricci tensor.
Choose local coordinate charts for $M$ where the 
Christoffel symbols $\Gamma_{ij}{}^k$ of $\nabla$ are locally constant. Then
\begin{eqnarray*}
&&R_{ijk}{}^l=\Gamma_{in}{}^l\Gamma_{jk}{}^n-\Gamma_{jn}{}^l\Gamma_{ik}{}^n,\qquad
\rho_{jk}=\Gamma_{in}{}^i\Gamma_{jk}{}^n-\Gamma_{jn}{}^i\Gamma_{ik}{}^n,\\
&&\rho_{s,jk}=\textstyle\frac12\{\Gamma_{in}{}^i\Gamma_{jk}{}^n-\Gamma_{jn}{}^i\Gamma_{ik}{}^n+
\Gamma_{in}{}^i\Gamma_{kj}{}^n-\Gamma_{kn}{}^i\Gamma_{ij}{}^n\}\,.
\end{eqnarray*}
The symmetric Ricci tensor gives an invariantly defined flat 
pseudo-Riemannian metric. 
Therefore, the transition functions between such charts are affine. The desired result follows since translations do
not affect $\Gamma$.
\end{proof}

There is an exceptional structure. Define $\Gamma_0\in\mathcal{Z}(2,0)$ and $C_0\subset\mathcal{Z}(2,0)$ by setting:
\begin{equation}\label{E1.a}\begin{array}{l}
\Gamma_{0;11}{}^1=-1,\Gamma_{0;12}{}^2=\Gamma_{0;21}{}^2=\Gamma_{0;22}{}^1=1,\,\,\Gamma_{0;ij}{}^k=0\text{ otherwise},\\[0.05in]
C_0:=\Gamma_0\cdot\operatorname{GL}(2,\mathbb{R})=\Gamma_0\cdot\operatorname{GL}^+(2,\mathbb{R})
\subset\mathcal{Z}(2,0)\,.
\end{array}\end{equation}
We will show presently in Lemma~\ref{L2.2} that $C_0$ is a closed set and, furthermore,
that any non-trivial fixed point of the action of $\operatorname{GL}^+(2,\mathbb{R})$ on
$\mathcal{Z}(p,q)$ belongs to this orbit; thus in particular, $\operatorname{GL}^+(2,\mathbb{R})$
acts without fixed points on $\mathcal{Z}(1,1)$ and on $\mathcal{Z}(0,2)$. To exclude this exceptional orbit, we define:
\begin{eqnarray*}
&&\tilde{\mathcal{Z}}(p,q):=\left\{\begin{array}{lll}
\mathcal{Z}(p,q)&\text{ if }(p,q)\ne(2,0)\\
\mathcal{Z}(p,q)-C_0&\text{ if }(p,q)=(2,0)
\end{array}\right\},\\
&&\tilde{\mathcal{W}}(p,q):=\left\{\begin{array}{ll}
\mathcal{W}(p,q)&\text{ if }(p,q)\ne(2,0)\\
\mathcal{W}(p,q)-C_0&\text{ if }(p,q)=(2,0)
\end{array}\right\},\\
&&\tilde{\mathfrak{Z}}^+(p,q):=\tilde{\mathcal{Z}}(p,q)/\operatorname{GL}^+(2,\mathbb{R}),
\qquad\quad\tilde{\mathfrak{Z}}(p,q):=\tilde{\mathcal{Z}}(p,q)/\operatorname{GL}(2,\mathbb{R}),\\
&&\tilde{\mathfrak{W}}^+(p,q):=\tilde{\mathcal{W}}(p,q)/\operatorname{GL}^+(2,\mathbb{R}),
\qquad\tilde{\mathfrak{W}}(p,q):=\tilde{\mathcal{W}}(p,q)/\operatorname{GL}(2,\mathbb{R})\,.
\end{eqnarray*}

\begin{theorem}\label{T1.3}
\ \begin{enumerate}
\item $\tilde{\mathfrak{Z}}^+(p,q)$  (resp. $\tilde{\mathfrak{W}}^+(p,q)$) is a smooth manifold of dimension $2$ (resp. $4$) and
$\tilde{\mathcal{Z}}(p,q)\rightarrow\tilde{\mathfrak{Z}}^+(p,q)$
(resp. $\tilde{\mathcal{W}}(p,q)\rightarrow \tilde{\mathfrak{W}}^+(p,q)$) is a $\operatorname{GL}^+(2,\mathbb{R})$ principal bundle.
\item $\tilde{\mathfrak{Z}}(p,q)$ (resp. $\tilde{\mathfrak{W}}(p,q)$) is a smooth manifold with boundary
(resp. without boundary) of dimension $2$ (resp. $4$). Furthermore,
$\tilde{\mathfrak{Z}}^+(p,q)\rightarrow\tilde{\mathfrak{Z}}(p,q)$
(resp. $\tilde{\mathfrak{W}}^+(p,q)\rightarrow\tilde{\mathfrak{W}}(p,q)$) is a ramified double cover.
\end{enumerate}\end{theorem}

Section~\ref{S2} is devoted to the proof of Assertion~2 and Section~\ref{S3} is devoted to the proof of Assertion~3.
The analysis is quite
different in the unoriented context as the full general linear group $\operatorname{GL}(2,\mathbb{R})$ has a 1-dimensional fixed point
set acting on $\mathcal{Z}(p,q)$ and a 2-dimensional fixed point set acting on $\mathcal{W}(p,q)$ for $(p,q)\in\{(2,0),(1,1),(0,2)\}$.
We will discuss the ramification sets in Section~\ref{S3.3} and in Section~\ref{S3.5} once the necessary notation has been developed. 
We must now consider the singular orbit $[\Gamma_0]$ when $(p,q)=(2,0)$. 
\begin{definition}\label{D1.4}
\rm Let $\mathbb{Z}_3:=\{1,\lambda,\lambda^2\}$ for $\lambda:=e^{2\pi\sqrt{-1}/3}$ be the cyclic group of order $3$ consisting of
the third roots of unity in $\mathbb{C}$. This group acts on $\mathbb{C}$ and on $\mathbb{C}^2$ by
complex multiplication. Since $\bar\lambda=\lambda^2$, complex conjugation defines a non-trivial automorphism
of the group $\mathbb{Z}_3$.
Thus complex multiplication by $\lambda$ and
complex conjugation generate a non-Abelian group of order $6$ which is isomorphic to the symmetric group $s_3$
which consists of the permutations of $3$ elements. This group acts naturally on $\mathbb{C}$ and on $\mathbb{C}^2$;
$\mathbb{Z}_3$ is a normal subgroup of $s_3$.
\end{definition}

We will establish the following result in Section~\ref{S4}.
\begin{theorem}\label{T1.5}
\ \begin{enumerate}
\item $\mathfrak{Z}^+(2,0)$ (resp. $\mathfrak{W}^+(2,0)$) is a smooth orbifold.
An orbifold coordinate chart near $[\Gamma_0]$ can be obtained by taking $\mathbb{C}$ (resp. $\mathbb{C}^2$)
modulo the action of $\mathbb{Z}_3$.
\item $\mathfrak{Z}(2,0)$ (resp. $\mathfrak{W}(2,0)$) is a smooth orbifold. An orbifold coordinate chart 
near $[\Gamma_0]$ can be defined by taking $\mathbb{C}$ (resp. $\mathbb{C}^2$)
modulo the action of $s_3$.
\end{enumerate}\end{theorem}

We note that although we shall work in the smooth category, these results continue to hold in the real analytic category.
We now turn to the question of invariants and, for the sake of completeness, present some results from M. Brozos-V\'{a}zquez
et al. \cite{BGGP16a}. We work in the torsion free setting. If $\Gamma\in\mathcal{Z}(p,q)$, let
$\text{dvol}$ be the oriented $2$-form that gives the pseudo-Riemannian volume element defined by $\rho$, let 
$\rho_{ij}^3:=\Gamma_{ik}{}^l\Gamma_{jl}{}^k$, let
$\psi_3:=\operatorname{Tr}_\rho\{\rho^3\}=\rho^{ij}\rho_{ij}^3$, let
$\Psi_3:=\det(\rho^3)/\det(\rho)$, and let
$\chi(\Gamma):=\rho(\Gamma_{ab}{}^b\Gamma_{ij}{}^k\rho^3_{kl}\rho^{ij}dx^a\wedge dx^l,\text{dvol})$.

\begin{theorem}\label{T1.6}
Let $\mathcal{M}=(M,\nabla)$ be a locally homogeneous torsion free surface of Type~$\mathcal{A}$ where
$\operatorname{Rank}\{\rho\}=2$.
\begin{enumerate}
\item $\psi_3$, $\Psi_3$, and $\chi$ are invariantly defined on $\mathcal{M}$ and does not depend on the particular choice of 
Type~$\mathcal{A}$ coordinates.
\item $\Xi(p,q):=(\psi_3,\Psi_3,\chi)$  is a 1-1 map from each $\mathfrak{Z}^+(p,q)$ to a closed surface
in $\mathbb{R}^3$ and provides a complete set of invariants in the oriented context.
\item $\Theta(p,q):=(\psi_3,\Psi_3)$ defines a 1-1 map from 
$\mathfrak{Z}(p,q)$ to a simply
connected closed subset $\mathfrak{V}(p,q)$ of $\mathbb{R}^2$ and provides a complete set of invariants in the unoriented context.
\end{enumerate}\end{theorem}

We wish to make this description of the moduli spaces as specific as possible. Consider the curves
$\sigma_\pm(t):=(\pm4t^2\pm\frac1{t^2}+2,4t^4\pm4t^2+2)$.
The curve $\sigma_+$ is smooth; the curve $\sigma_-$ has a cusp at $(-2,1)$ when $t=\frac1{\sqrt2}$. This
corresponds to the structure $\Gamma_0$ given above in Equation~(\ref{E1.a}).
The curves $\sigma_\pm$ divide the plane into 3 open regions $\mathfrak{O}(2,0)$, $\mathfrak{O}(1,1)$, and $\mathfrak{O}(0,2)$. 
The region  $\mathfrak{O}(2,0)$ lies in the second quadrant and is bounded on the right
by $\sigma_-$; the region $\mathfrak{O}(0,2)$ lies in the first quadrant and is bounded on the left by $\sigma_+$;
the region $\mathfrak{O}(1,1)$ lies in between and is bounded on the left by $\sigma_-$ and on the right by $\sigma_+$.
The regions $\mathfrak{D}(p,q)=\Theta(p,q)\{\mathfrak{Z}(p,q)\}$ discussed in Theorem~\ref{T1.6} are the closure of 
the open sets $\mathfrak{O}(p,q)$. 
We picture below $\mathfrak{D}(2,0)$ in blue,
$\mathfrak{D}(1,1)$ in white, and $\mathfrak{D}(0,2)$ in red:
\goodbreak\centerline{Figure 1.1}
\vbox{\par\centerline{\includegraphics[height=3.5cm,keepaspectratio=true]{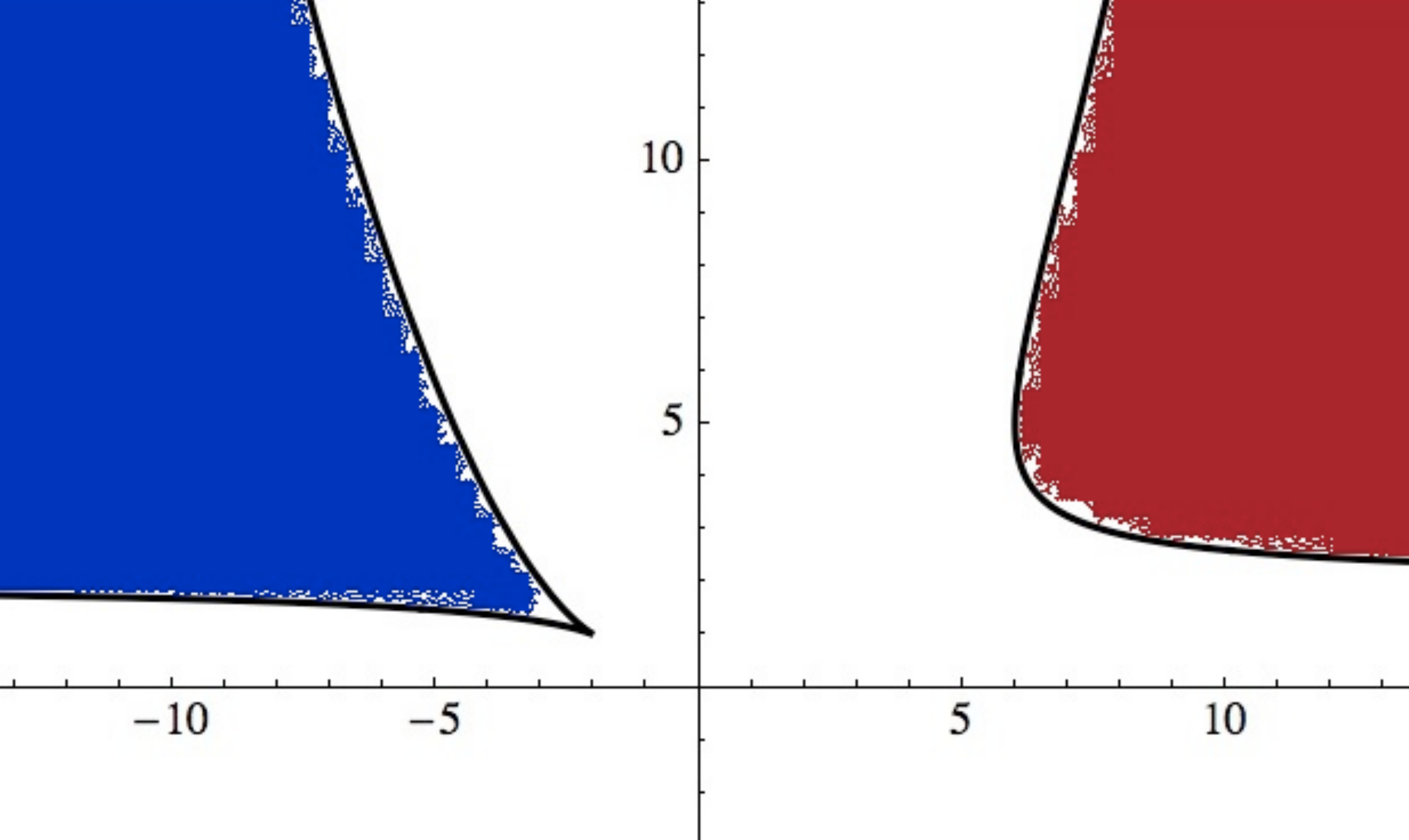}}}
\medbreak\noindent
Note that although
$\Theta(p,q)$ is 1-1 on $\mathfrak{Z}(p,q)$,
$\Theta(0,2)(\mathfrak{Z}(0,2))$ intersects $\Theta(1,1)(\mathfrak{Z}(1,1))$
along their common boundary $\sigma_+$ and $\Theta(2,0)(\mathfrak{Z}(2,0))$ intersects $\Theta(1,1)(\mathfrak{Z}(1,1))$
along their common boundary $\sigma_-$. This does not mean that $\mathfrak{Z}(2,0)$ or $\mathfrak{Z}(0,2)$ intersects
$\mathfrak{Z}(1,1)$ nor does it mean that $\Theta(p,q)$ is not 1-1 on these sets separately. Although it appears from the picture
that $\mathfrak{Z}(1,1)$ has a cusp singularity, Theorem~\ref{T1.3} shows that this is not the case and
the apparent cusp is an artifact of the parametrization in terms of the given invariants. On the other hand, $\mathfrak{Z}(2,0)$
really does have a singularity at $\Gamma_0$ by Theorem~\ref{T1.5}.

\section{Proper and fixed point free group actions}\label{S2}
\subsection{Principal bundles}
Let $G$ be a Lie group which acts smoothly on a manifold $N$. One says the action of $G$ on $N$ is {\it proper}
if given $P_n\in N$ and $g_n\in G$ with $P_n\rightarrow P$ and $g_nP_n\rightarrow\tilde P$, there is a subsequence
so $g_{n_k}\rightarrow g\in G$. If $P\in N$, let $G_P:=\{g\in G:gP=P\}$ be the isotropy subgroup. One says the action is
{\it fixed point free} if $G_P=\{\operatorname{id}\}$ for any $P\in N$.

\begin{lemma}\label{L2.1} Let $G$ be a Lie group acting smoothly on a smooth manifold $N$. 
\begin{enumerate}
\item If $G_P$ is a discrete subgroup of $G$, then $\Psi_P:g\rightarrow g\cdot P$ is a smooth immersion of $G$ into $N$.
\item If the action of $G$ on $N$ is proper and fixed point free,
then $N/G$ is a smooth manifold and
$N\rightarrow N/G$ is a principal $G$ bundle.
\end{enumerate}
\end{lemma}

\begin{proof} Although this result is well-known, we sketch the proof briefly
to introduce the notation we shall need subsequently. Let $\mathfrak{g}:=T_{\operatorname{id}}(G)$ be
the Lie algebra of $G$. Let $\exp:g\rightarrow G$ be the exponential map.
If $0\ne\xi\in\mathfrak{g}$, let $\gamma_{\xi,P}(t):=\exp(t\xi)\cdot P$.
We wish to show that $\dot\gamma_{\xi,P}(0)\ne0$; we suppose to the contrary that $\dot\gamma_{\xi,P}(0)=0$
\and argue for a contradiction.
Since $\exp((t+s)\xi)=\exp(t\xi)\exp(s\xi)$, one has that
$\gamma_{\xi,P}(t+s)=\exp(t\xi)\cdot\gamma_{\xi,P}(s)$ and $\dot\gamma_{\xi,P}(t)=\exp(t\xi)_*\dot\gamma_{\xi,P}(0)=0$. 
Since $\dot\gamma_{\xi,P}(t)=0$ for all $t$, 
$\gamma_{\xi,P}(t)$ is the constant map. 
This contradicts the hypothesis that $G_P$ is a discrete subgroup of $G$.
Thus we conclude that $\dot\gamma_{\xi,P}(0)\ne0$ so $d\Psi_P(0)$
is an injective map from $\mathfrak{g}$ to $T_PN$. This shows that
$\Psi_P$ is an immersion near $g=\operatorname{id}$.
We use the group action to see that $\Psi_P$ is an immersion near any point of $G$.

Suppose the action is fixed point free and proper. Let $\Sigma$ be the germ of a submanifold which is transverse to the orbit
$\mathcal{O}_P:=G\cdot P$ at $P$; $\Sigma$ is called a {\it slice}.  Let $\Phi(g,s):=g\cdot s$ for $g\in G$ and  $s\in\Sigma$.
Since $\Sigma$ is transverse to $\mathcal{O}_P$ at $P$, Assertion~1 implies that $\Psi_*$ is an isomorphism from
$T_{\operatorname{id}}(G)\times T_P(\Sigma)$ to $T_PN$. We use the transitive group action on $\mathcal{O}_P$ 
to see that $\Phi$ is an immersion from $G\times\Sigma\rightarrow N$. 
We wish to show that $\Phi$ is an embedding if we restrict to a sufficiently small neighborhood of $P$ in $\Sigma$. Suppose, to the contrary, that no such neighborhood exists. This implies that there exists
$(g_n,P_n,\tilde g_n,\tilde P_n)$ so that
$$\begin{array}{llll}
P_n\in\Sigma,& \tilde P_n\in\Sigma, &g_n\in G,&\tilde g_n\in G,\\
P_n\rightarrow P,&\tilde P_n\rightarrow P,&g_nP_n=\tilde g_n\tilde P_n,&
(g_n,P_n)\ne(\tilde g_n,\tilde P_n)\,.
\end{array}$$
By replacing $g_n$ by $\tilde g_n^{-1}g_n$, we may assume $\tilde g_n=1$. 
Since the action is proper, we can choose a subsequence of the
$g_n$ which converges to $g$.
Since $P_n\rightarrow P$ and $\tilde P_n\rightarrow P$, we have $gP=P$. Since the action is fixed point free,
$g=\operatorname{id}$ and thus $g_n\rightarrow\operatorname{id}$. This
is contradicts the fact that $G\cdot\Sigma\rightarrow N$ is an immersion. Thus if we restrict $\Sigma$ suitably,
$G\times\Sigma$ may be identified with a suitable closed subset of $N$ which is a neighborhood of $G\cdot P$. 
This gives the requisite principal bundle charts;
the projection of $\Sigma$ to $N/G$ is 1-1 and gives the appropriate charts on the quotient $N/G$; the transition
maps between these charts are smooth.\end{proof}

\subsection{The action of $\operatorname{GL}^+(2,\mathbb{R})$ on $\mathcal{Z}(p,q)$}
If $\Gamma\in\mathcal{Z}(p,q)$, set:
\medbreak\qquad\qquad
$\operatorname{SO}^+(\varrho_{s,\Gamma}):
=\{T\in\operatorname{GL}^+(2,\mathbb{R}):T^*\varrho_{s,\Gamma}=\varrho_{s,\Gamma}\}$,
\medbreak\qquad\qquad
$G^+_\Gamma:=\{T\in\operatorname{GL}^+(2,\mathbb{R}):T\Gamma=\Gamma\}\subset\operatorname{SO}^+(\rho_{s,\Gamma})$. 
\begin{lemma}\label{L2.2}
Let $p+q=2$. Let $\Gamma_0$ be the exceptional structure of Equation~(\ref{E1.a}).
\begin{enumerate}
\item The action of  $\operatorname{GL}(2,\mathbb{R})$ on $\mathcal{W}(p,q)$ and on $\mathcal{Z}(p,q)$ is proper
\item $C_0:=\operatorname{GL}^+(2,\mathbb{R})\cdot\Gamma_0$ is a closed subset of $\mathcal{Z}(2,0)$ and of $\mathcal{W}(2,0)$.
\item If $\Gamma\in\mathcal{Z}(p,q)$ satisfies $G_\Gamma^+\ne\{\operatorname{id}\}$, then 
$(p,q)=(2,0)$ and $\Gamma\in C_0$.
\end{enumerate}
\end{lemma}

\begin{proof} We wish to prove the action is proper. Since $\operatorname{GL}^+(2,\mathbb{R})$ is a subgroup of finite
index in $\operatorname{GL}(2,\mathbb{R})$, it suffices to prove the action of $\operatorname{GL}^+(2,\mathbb{R})$
on $\mathcal{W}(p,q)$ is proper; it will then follow that the action of $\operatorname{GL}(2,\mathbb{R})$ is proper.
Restricting to $\mathcal{Z}(p,q)$ then yields a proper action as well.

Suppose that there exists $g_n\in\operatorname{GL}^+(2,\mathbb{R})$ and
that there exist $\Gamma_n$, $\Gamma$, and $\tilde\Gamma$ in $\mathcal{W}(p,q)$ so that $\Gamma_n\rightarrow\Gamma$
and $g_n\Gamma_n\rightarrow\tilde\Gamma$. To prove Assertion~1, we
must extract a subsequence $g_{n_k}$ which is convergent. Let $\varrho:=\rho_{s,\tilde\Gamma}$; by hypothesis,
$\varrho$ is a non-degenerate symmetric bilinear form of signature $(p,q)$. Let $\tilde\Gamma_n:=g_n\Gamma_n$. Since
$\tilde\Gamma_n\rightarrow\tilde\Gamma$, we have $\rho_{s,\tilde\Gamma_n}\rightarrow\varrho$. For $n$ large, we
can apply the Gram-Schmidt process to find $h_n\in\operatorname{GL}^+(2,\mathbb{R})$ with $h_n\rightarrow\operatorname{id}$
and $h_n\rho_{s.\tilde\Gamma_n}=\varrho$; it is necessary to take $n$ large to ensure $\rho_{s,\tilde\Gamma_n}$ is
close to $\varrho$ and thus we are not dividing by zero when applying the Gram-Schmidt process as $\varrho$ could be indefinite.
We then have $\rho_{h_ng_n\Gamma_n}=\varrho$ and $h_ng_n\Gamma_n\rightarrow\tilde\Gamma$. Since extracting
a convergent sequence from $h_ng_n$ is equivalent to extracting a convergent sequence from $g_n$, we may assume
without loss of generality that $\rho_{s,g_n\Gamma_n}=\varrho$ for all $n$. 

We have also that $\tilde\Gamma_n\rightarrow\tilde\Gamma$ and $g_n^{-1}\tilde\Gamma_n\rightarrow\Gamma$.
Choose $h$ so $h\rho_{\Gamma}=\varrho$. Extracting a convergent sequence from amongst the $g_n$ is equivalent
to extracting a convergent sequence from amongst the $hg_n^{-1}$. Thus we may assume $\rho_{s,\Gamma}=\varrho$
as well without altering the normalizations $\rho_{\tilde\Gamma_n}=\rho_{\tilde\Gamma}=\varrho$. 
We clear the previous notation and apply the Gram-Schmidt process to find $h_n\in\operatorname{GL}^+(2,\mathbb{R})$
so $h_n\rightarrow\operatorname{id}$ and $h_n\rho_{g_n^{-1}\tilde\Gamma_n}=\varrho$. Replacing the $g_n^{-1}$ by $h_ng_n^{-1}$,
we may assume 
$$
\rho_\Gamma=\varrho\quad\rho_{\tilde\Gamma}=\varrho,\quad\rho_{\Gamma_n}=\varrho,\quad\rho_{g_n\Gamma_n}=\varrho\,.
$$
This implies that the $g_n\in\operatorname{SO}^+(\varrho)$. If $(p,q)=(2,0)$ or $(p,q)=(0,2)$, then $\operatorname{SO}^+(\varrho)$
is a compact Lie group and we can extract a convergent subsequence. We therefore assume $(p,q)=(1,1)$.

Choose the basis for $\mathbb{R}^2$ so $\varrho=dx^1\otimes dx^2+dx^2\otimes dx^1$. Since $g_n\in\operatorname{SO}^+(\varrho)$,
we may conclude $g_n(x^1,x^1)=(a_nx^1,a_n^{-1}x^2)$. If $|a_n|$ and $|a_n^{-1}|$ remain uniformly bounded, we can extract
a convergent subsequence. Thus by interchanging the roles of $x^1$ and $x^2$ if necessary, we may
assume that $|a_n|\rightarrow\infty$. We argue for a contradiction. We express
\begin{equation}\label{E2.a}
\begin{array}{l}
(g_n\Gamma)_{n,ij}{}^k=a_n^{\epsilon_{ijk}}\Gamma_{n,ij}{}^k\text{ for }\\[0.05in]
\epsilon_{ijk}:=\delta_{1i}-\delta_{2i}+\delta_{1j}-\delta_{2j}-\delta_{1k}+\delta_{2k}\in\{\pm1,\pm3\}\,.
\end{array}\end{equation}
Since $(g_n\Gamma_n)_{ij}{}^k$ converges to $\tilde\Gamma_{ij}{}^k$ and $\Gamma_{n,ij}{}^k$ converge
to $\Gamma_{ij}{}^k$, we have $\Gamma_{ij}{}^k=0$ for $\epsilon_{ijk}>0$.
Consequently,
$\Gamma_{11}{}^1=0$, $\Gamma_{11}{}^2=0$, $\Gamma_{12}{}^2=0$ and $\Gamma_{21}{}^2=0$.
We may then compute that $\rho_{11}=\rho_{12}=\rho_{21}=0$ which is impossible. This contradiction completes the proof of Assertion~1.
Since the action by $\operatorname{GL}^+(2,\mathbb{R})$ is proper, any orbit is closed. This establishes Assertion~2. 

To prove Assertion~3, we examine the isotropy group. Assume that there exists $\operatorname{id}\ne g\in G_\Gamma^+$. 
Since $g\Gamma=\Gamma$, $g\rho_{s,\Gamma}=\rho_{s,\Gamma}$ so $g\in\operatorname{SO}^+(\rho_{s,\Gamma})$.
If $\rho_{s,\Gamma}$ has indefinite signature, then we can choose the coordinates
so $gx^1=ax^1$ and $gx^2=a^{-1}x^1$ for $a\ne1$. Adopt the notation of Equation~(\ref{E2.a}).
We have $(g\Gamma)_{ij}{}^k=a^{\epsilon_{ijk}}\Gamma_{ij}{}^k$. Since $\epsilon_{ijk}$ is odd and $a\ne1$,
setting $g\Gamma=\Gamma$ implies all the $\Gamma_{ij}{}^k$ vanish which is false.
Thus the action is fixed point free in signature $(1,1)$.

Suppose the signature is definite. Introduce a complex basis
\begin{equation}\label{E2.b}
\begin{array}{llll}
f_1:=e_1+\sqrt{-1}e_2,&f_2:=e_1-\sqrt{-1}e_2,&gf_1=\alpha f_1,&gf_2=\bar\alpha f_2,\\
f^1:=\frac12(e^1-\sqrt{-1}e^2),&f^2:=\frac12(e^1+\sqrt{-1}e^2)&gf^1=\bar\alpha f^1,&gf^2=\alpha f^2
\end{array}\end{equation}
for $\alpha=e^{\sqrt{-1}\theta}$ appropriately chosen to describe the rotation involved. Since $\operatorname{id}\ne g$,
$\alpha\ne1$. Let $\tilde\Gamma_{ij}{}^k$ reflect the Christoffel symbols relative to this complex basis. Adopt the notation
of Equation~(\ref{E2.a}). We
then have $(g\tilde\Gamma)_{ij}{}^k=\alpha^{\epsilon_{ijk}}\tilde\Gamma_{ij}{}^k$.
This implies $\tilde\Gamma_{ij}{}^k=0$ if $\epsilon_{ijk}=\pm1$
and furthermore, that $\alpha^3=0$. Thus there is a complex number $\beta$ so
$$\begin{array}{llll}
\tilde\Gamma_{11}{}^1=0,&\tilde\Gamma_{11}{}^2=\beta,&\tilde\Gamma_{12}{}^1=0,&\tilde\Gamma_{12}{}^2=0,\\
\tilde\Gamma_{21}{}^1=0,&\tilde\Gamma_{21}{}^2=0,&\tilde\Gamma_{22}{}^1=\bar\beta,&\tilde\Gamma_{22}{}^2=0.
\end{array}$$
Since $\tilde\Gamma_{12}{}^1=\tilde\Gamma_{21}{}^1$
and $\tilde\Gamma_{12}{}^2=\tilde\Gamma_{21}{}^2$, $\tilde\Gamma$ and hence $\Gamma$ is torsion free.
By performing a coordinate rotation, we can assume that $\beta$ is real. 
We obtain
\begin{eqnarray*}
0&=&\tilde\Gamma_{11}{}^1=\textstyle\frac12\{\Gamma_{11}{}^1+2\Gamma_{12}{}^2-\Gamma_{22}{}^1
+\sqrt{-1}(-\Gamma_{11}{}^2+2\Gamma_{12}{}^1+\Gamma_{22}{}^2)\},\\
0&=&\tilde\Gamma_{12}{}^1=\textstyle\frac12\{\Gamma_{11}{}^1+\Gamma_{22}{}^1\qquad\hspace{15.5pt}
+\sqrt{-1}(-\Gamma_{11}{}^2-\Gamma_{22}{}^2)\},\\
\beta&=&\tilde\Gamma_{11}{}^2=\textstyle\frac12\{\Gamma_{11}{}^1-2\Gamma_{12}{}^2-\Gamma_{22}{}^1
+\sqrt{-1}(+\Gamma_{11}{}^2+2\Gamma_{12}{}^1-\Gamma_{22}{}^2)\}\,.
\end{eqnarray*}
We solve these equations to obtain the relations
$$\begin{array}{lll}
\Gamma_{11}{}^1=\frac12\beta,&\Gamma_{22}{}^1=-\frac12\beta,&\Gamma_{12}{}^2=\Gamma_{21}{}^2=-\frac12\beta,\\ [0.05in]
\Gamma_{11}{}^2=0,&\Gamma_{22}{}^2=0,&\Gamma_{12}{}^1=\Gamma_{21}{}^1=0.
\end{array}$$
By rescaling the coordinate system, we can ensure $\frac12\beta=-\frac1{\sqrt2}$ and obtain the structure of Equation~(\ref{E1.a}):
$$\begin{array}{lll}
\Gamma_{11}{}^1=-\frac1{\sqrt2},&\Gamma_{22}{}^1=\frac1{\sqrt2},&\Gamma_{12}{}^2=\Gamma_{21}{}^2=\frac1{\sqrt2},\\ [0.05in]
\Gamma_{11}{}^2=0,&\Gamma_{22}{}^2=0,&\Gamma_{12}{}^1=\Gamma_{21}{}^1=0.
\end{array}$$
One may then compute that
$\rho_s=\operatorname{diag}(-1,-1)$ which has signature $(2,0)$. This completes the proof of Assertion~3. 
\end{proof}

\subsection{The proof of Theorem~\ref{T1.3} Assertion~1}
This is a direct consequence of Lemma~\ref{L2.1} and of Lemma~\ref{L2.2}.\hfill\qed

\section{The projections $\tilde{\mathfrak{Z}}^+(p,q)\rightarrow\tilde{\mathfrak{Z}}(p,q)$
and $\tilde{\mathfrak{W}}^+(p,q)\rightarrow\tilde{\mathfrak{W}}(p,q)$}\label{S3}

\subsection{The nature of the ramification set}
Fix an element $T\in\operatorname{GL}(2,\mathbb{R})$ of order $2$ with $\det(T)=-1$.
If $g\in\operatorname{GL}^+(2,\mathbb{R})$, twist the standard action by defining $g*\Gamma:=TgT^{-1}\Gamma$.
The map $\Gamma\rightarrow T\Gamma$ then intertwines these actions and consequently $T$
descends to define a map $[T]$ of order $2$ on the moduli spaces which defines an action of $\mathbb{Z}_2$ so that
$$
\mathfrak{Z}(p,q)=\mathfrak{Z}^+(p,q)/\mathbb{Z}_2\text{ and }
\mathfrak{W}(p,q)=\mathfrak{W}^+(p,q)/\mathbb{Z}_2\,.
$$
Denote the fixed point sets by:
\begin{eqnarray*}
&&\widetilde{\mathfrak{FZ}}^+(p,q):=\{[\Gamma]\in\tilde{\mathfrak{Z}}^+(p,q):[T\Gamma]=[\Gamma]\},\\
&&\widetilde{\mathfrak{FW}}^+(p,q):=\{[\Gamma]\in\tilde{\mathfrak{W}}^+(p,q):[T\Gamma]=[\Gamma]\}\,.
\end{eqnarray*}
On the complement of the fixed point set, the projection from the oriented moduli space to the unoriented moduli space is a double
cover and the $\mathbb{Z}_2$ quotient inherits a natural smooth structure. We examine the fixed point sets as follows:

\begin{lemma}\label{L3.1}
\ \begin{enumerate}
\item $\widetilde{\mathfrak{FZ}}^+(p,q)$ is a smooth 1-dimensional submanifold of $\tilde{\mathfrak{Z}}^+(p,q)$.
\item $\widetilde{\mathfrak{FW}}^+(p,q)$ is a smooth 2-dimensional submanifold of $\tilde{\mathfrak{W}}^+(p,q)$.
\end{enumerate}\end{lemma}

\begin{proof} We have excluded the exceptional orbit of $\Gamma_0$ and restrict to the oriented
moduli spaces; $[T]$ is then smooth. By averaging over the
action of $\mathbb{Z}_2$, we can put smooth Riemannian metrics on the oriented moduli spaces. The fixed point sets
are then totally geodesic submanifolds of $\tilde{\mathfrak{Z}}^+(p,q)$ and $\tilde{\mathfrak{W}}^+(p,q)$, respectively.
In particular they are smooth.

Let $Te_1=-e_1$ and $Te_2=e_2$. Then $T\Gamma=\Gamma$ implies $\Gamma=\Gamma(a,b,c,d)$ where:
$$\begin{array}{llll}
\Gamma_{11}{}^1=0,&\Gamma_{11}{}^2=a,&\Gamma_{12}{}^1=b,&\Gamma_{12}{}^2=0,\\
\Gamma_{21}{}^1=c,&\Gamma_{21}{}^2=0,&\Gamma_{22}{}^1=0,&\Gamma_{22}{}^2=d,
\end{array}$$
We exclude the exceptional orbit $C_0:=\Gamma_0\cdot\operatorname{GL}(2,\mathbb{R})$ and define: 
\begin{eqnarray*}
&&\widetilde{\mathcal{FZ}}(p,q):=\{\Gamma(a,c,c,d)\in\tilde{\mathcal{Z}}(p,q)\}\cap
C_0^c,\\
&&\widetilde{\mathcal{FW}}(p,q):=\{\Gamma(a,b,c,d)\in\tilde{\mathcal{W}}(p,q)\}\cap
C_0^c\,;
\end{eqnarray*}
$\widetilde{\mathcal{FZ}}(p,q)$ is an open subset of $\mathbb{R}^3$ and
$\widetilde{\mathcal{FW}}(p,q)$ is an open subset of $\mathbb{R}^4$ since
\begin{equation}\label{E3.a}
\rho_\Gamma=\left(\begin{array}{cc}a(d-c)&0\\0&b(c-d)\end{array}\right)\,.
\end{equation}
If $T\Gamma_1=\Gamma_1$, $T\Gamma_2=\Gamma_2$, and $g\Gamma_1=\Gamma_2$ for
$g\in\operatorname{GL}^+(2,\mathbb{R})$ for $\Gamma_i\in\widetilde{\mathcal{FW}}(p,q)$, then
$$
TgT\Gamma_1=Tg\Gamma_1=T\Gamma_2=\Gamma_2\text{ so }g^{-1}TgT\Gamma_1=\Gamma_1\,.
$$
Since we have excluded the exceptional orbit from consideration,
$\operatorname{GL}^+(2,\mathbb{R})$ acts without fixed points and $Tg=gT$. Thus the structure group in question is given by
$
G_0:=\{g\in\operatorname{GL}^+(2,\mathbb{R}):Tg=gT\}\,.
$
We compute:
\begin{eqnarray*}
&&Tg=\left(\begin{array}{rr}-1&0\\0&1\end{array}\right)\left(\begin{array}{rr}a_{11}&a_{12}\\a_{21}&a_{22}\end{array}\right)
=\left(\begin{array}{rr}-a_{11}&-a_{12}\\a_{21}&a_{22}\end{array}\right)\\
&=&gT=\left(\begin{array}{rr}a_{11}&a_{12}\\a_{21}&a_{22}\end{array}\right)\left(\begin{array}{rr}-1&0\\0&1\end{array}\right)
=\left(\begin{array}{rr}-a_{11}&a_{12}\\-a_{21}&a_{22}\end{array}\right)\,.
\end{eqnarray*}
This implies $a_{12}=a_{21}=0$ so
$$
G_0=\left\{g=\left(\begin{array}{cc}a_{11}&0\\0&a_{22}\end{array}\right):\det(g)>0\right\}
$$
is two dimensional.
By Lemma~\ref{L2.2}, the action of $G_0$ on $\widetilde{\mathcal{FZ}}(p,q)$ and on $\widetilde{\mathcal{FW}}(p,q)$ 
is proper and without fixed points. Thus the projections to the oriented moduli spaces are principal $G_0$ bundles. 
It then follows that
\medbreak\hfill
$\dim\{\widetilde{\mathfrak{FZ}}^+(p,q)\}=3-2=1$ and $\dim\{\widetilde{\mathfrak{FW}}^+(p,q)\}=4-2=2$.
\hfill\vphantom{.}
\end{proof}

\subsection{The proof of Theorem~\ref{T1.3} Assertion~2}
Let $F$ be a component of the fixed point set and let $\nu$ be the normal bundle of $F$. The $\mathbb{Z}_2$
action acts was multiplication by $-1$ on $\nu$. If $F$ has
codimension $1$, this replaces the open fiber intervals $(-\varepsilon,\varepsilon)$ by half open intervals $[0,\varepsilon^2)$
and ensures that $F$ becomes a part of the boundary of the unoriented moduli
space. If $F$ has codimension $2$, then $\nu$ is a 2-plane bundle. The analysis is local so we can assume $\nu$ is a complex
line bundle $L$. Identifying antipodal points is equivalent to passing to $L^2$ and we obtain a smooth structure on the quotient where the
double cover ramifies over $F$.\hfill\qed

\subsection{The ramification set for the projection $\widetilde{\mathfrak{FZ}}^+(p,q)\rightarrow\widetilde{\mathfrak{FZ}}(p,q)$}\label{S3.3}
Let $P^+$ belong to the fixed point set $\widetilde{\mathfrak{FZ}}^+(p,q)$. 
We may choose real local coordinates $(\eta^1,\eta^2)\in(-\epsilon,\epsilon)\times(-\epsilon,\epsilon)$
which are centered at $P$ so that $\widetilde{\mathfrak{FZ}}^+(p,q)$ is defined by setting the real fiber coordinate $\eta^2=0$. Local
coordinates for the corresponding point $P$ in the boundary of $\widetilde{\mathfrak{FZ}}(p,q)$ are then
given by taking $(\eta^1,\eta^2)\in(-\epsilon,\epsilon)\times[0,\epsilon^2)$. The ramified double cover is the fold
singularity given by $(\eta^1,\eta^2)\rightarrow(\eta^1,(\eta^2)^2)$.

\subsection{The number of boundary components in $\mathfrak{Z}(p,q)$} Let $(p,q)=(1,1)$ or $(p,q)=(0,2)$ so the
exceptional orbit plays no role. The structure group $G_0$ has two arc components corresponding to $\{a_{11}>0,a_{22}>0\}$ and
$\{a_{11}<0,a_{22}<0\}$. Thus the number of boundary components in $\mathfrak{Z}(p,q)$ is half the number
of components in $\mathcal{FZ}(p,q)$. We apply Equation~(\ref{E3.a}) setting $c=b$
so $\rho=\operatorname{diag}(a(d-b),b(b-d))$. The Ricci tensor is positive definite if
$a(d-b)>0$ and $b(b-d)>0$. This gives rise to two components $\{a>0,b<0,d>b\}$ and $\{a<0,b>0,d<b\}$.
Thus the boundary of $\mathfrak{Z}(0,2)$ is connected; this is in agreement with Figure 1.1. On the other hand,
the Ricci tensor is indefinite if $\{a(d-b)>0,b(b-d)<0\}$ or $\{a(d-b)<0,b(b-d)>0\}$. This gives rise to four
components $\{a>0,d>b,b>0\}$ or $\{a<0,d<b,b<0\}$ or $\{a>0,d>b,b<0\}$ or $\{a<0,d<b,b>0\}$. Thus
$\mathfrak{Z}_{1,1}$ has 2 boundary components. This again is in agreement
with Figure 1.1. Finally, if we ignore the effect of the singular orbit, which we will treat in the next section,
the same analysis shows $\mathfrak{Z}(2,0)$ has one boundary
component.

\subsection{The ramification set for the projection $\widetilde{\mathfrak{FW}}^+(p,q)\rightarrow\widetilde{\mathfrak{FW}}(p,q)$}\label{S3.5}
Let $B_\epsilon(0):=\{\eta\in\mathbb{C}:|\eta|<\epsilon\}$ be the open ball of radius $\epsilon$ in $\mathbb{C}$. Let $P^+$ belong to the fixed point set $\widetilde{\mathfrak{FW}}^+(p,q)$.
We may choose local complex coordinates $(\eta^1,\eta^2)\in B_\epsilon\times B_\epsilon$
which are centered at $P^+$ so that locally $\widetilde{\mathfrak{FW}}^+(p,q)$ is defined by setting
the complex fiber coordinate $\eta^2=0$. Local
coordinates for the corresponding point $P$ in $\widetilde{\mathfrak{FW}}(p,q)$ are then
given by taking $(\eta^1,\eta^2)\in B_\epsilon\times B_\epsilon$ and the ramified double cover is then the quadratic
singularity given by $(\eta^1,\eta^2)\rightarrow(\eta^1,(\eta^2)^2)$.

\section{The orbifold structure near the singular orbit $[\Gamma_0]$}\label{S4}
\subsection{Complex coordinates}
We adopt the notation of Equation~(\ref{E2.b}) and complexity:
$$\begin{array}{ll}
f_1:=e_1+\sqrt{-1}e_2,&f_2:=e_1-\sqrt{-1}e_2,\\
f^1:=\frac12(e^1-\sqrt{-1}e^2),&f^2:=\frac12(e^1+\sqrt{-1}e^2)\,.
\end{array}$$
We identify $\mathbb{R}^8=\mathbb{C}^4$
and define coordinates $\vec\alpha:=(\alpha_{11}{}^1,\alpha_{11}{}^2,\alpha_{12}{}^1, \alpha_{12}{}^2)\in\mathbb{C}^4$ 
on $\mathcal{W}(2,0)$ by defining $\tilde\Gamma(\vec\alpha)$ to be:
$$\begin{array}{llll}
\tilde\Gamma_{11}{}^1=\alpha_{11}{}^1,&\tilde\Gamma_{11}{}^2=\alpha_{11}{}^2,&
\tilde\Gamma_{12}{}^1=\alpha_{12}{}^1,&\tilde\Gamma_{12}{}^2=\alpha_{12}{}^2,\\
\tilde\Gamma_{22}{}^2=\bar\alpha_{11}{}^1,&\tilde\Gamma_{22}{}^1=\bar\alpha_{11}{}^2,&
\tilde\Gamma_{21}{}^2=\bar\alpha_{12}{}^1,&\tilde\Gamma_{21}{}^1=\bar\alpha_{12}{}^2.
\end{array}$$
The singular orbit is then $\tilde\Gamma(0,1,0,0)$;
where we suppress normalizing constant of $1/{\sqrt2}$ as it plays no role.
Similar coordinates on $\mathcal{Z}(2,0)$ taking values in $\mathbb{C}^3$ are obtained by imposing the single condition
$$
\alpha_{12}{}^1=\tilde\Gamma_{12}{}^1=\tilde\Gamma_{21}{}^1=\bar\alpha_{12}{}^2\,.
$$
We then have automatically
$\tilde\Gamma_{12}{}^2=\alpha_{12}{}^2=\bar\alpha_{12}{}^1=\tilde\Gamma_{21}{}^2$.

\subsection{A complex representation of the general linear group}
If $T\in\operatorname{GL}^+(2,\mathbb{R})$, then $T=T_{\beta_1,\beta_2}$ for $|\beta_1|^2-|\beta_2|^2>0$ where
$$\begin{array}{ll}
T_{\beta_1,\beta_2}f_1=\beta_1f_1+\beta_2f_2,&T_{\beta_1,\beta_2}f_2=\bar\beta_2f_1+\bar\beta_1f_2,\\[0.05in]
T_{\beta_1,\beta_2}f^1=\frac1{|\beta_1|^2-|\beta_2|^2}(\bar\beta_1f^1-\bar\beta_2f^2),&
T_{\beta_1,\beta_2}f^2=\frac1{|\beta_1|^2-|\beta_2|^2}(-\beta_2f^1+\beta_1f^2).
\end{array}$$
We wish to compute the tangent space to the orbit $C_0:=\operatorname{GL}^+(2,\mathbb{R})\cdot\Gamma_0$. We consider the
two curves $T_{1+t\beta,0}$ and $T_{1,t\beta}$.
$$\begin{array}{ll}
\{\partial_tT_{1+t\beta_1,0}\Gamma_0|_{t=0}\}_{11}{}^1=0,&
\{\partial_tT_{1+t\beta_1,0}\Gamma_0|_{t=0}\}_{11}{}^2=3\beta_1,\\
\{\partial_tT_{1+t\beta_1,0}\Gamma_0|_{t=0}\}_{12}{}^1=0,&
\{\partial_tT_{1+t\beta_1,0}\Gamma_0|_{t=0}\}_{12}{}^2=0,\\
\{\partial_tT_{1,t\beta_2}\Gamma_0|_{t=0}\}_{11}{}^1=-\bar\beta_2,&
\{\partial_tT_{1,t\beta_2}\Gamma_0|_{t=0}\}_{11}{}^2=0,\\
\{\partial_tT_{1,t\beta_2}\Gamma_0|_{t=0}\}_{12}{}^1=0,&
\{\partial_tT_{1,t\beta_2}\Gamma_0|_{t=0}\}_{12}{}^2=\bar\beta_2.
\end{array}$$
Thus a transversal slice $s_{\mathcal{W}}(\alpha_1,\alpha_2)$ to 
$C_0$ in $\mathcal{W}(2,0)$ can be taken to be:
$$\begin{array}{llll}
s_{\mathcal{W},11}{}^1:=0,& s_{\mathcal{W},11}{}^2:=1,&
s_{\mathcal{W},12}{}^1:=\bar\alpha_2,&s_{\mathcal{W},12}{}^2:=\alpha_1,\\[0.05in]
s_{\mathcal{W},22}{}^2:=0,& s_{\mathcal{W},22}{}^1:=1,&
s_{\mathcal{W},21}{}^2:=\alpha_2,&s_{\mathcal{W},21}{}^1:=\bar\alpha_1.
\end{array}$$
In defining the transversal slice $s_{\mathcal{Z}}(\alpha)$ over $\mathcal{Z}(2,0)$, we set $\alpha:=\alpha_1=\alpha_2$ to ensure that
$\tilde\Gamma_{12}{}^1=\tilde\Gamma_{21}{}^1$ and $\tilde\Gamma_{12}{}^2=\tilde\Gamma_{21}{}^2$:
$$\begin{array}{llll}
s_{\mathcal{Z},11}{}^1:=0,& s_{\mathcal{Z},11}{}^2:=1,&
s_{\mathcal{Z},12}{}^1:=\bar\alpha,&s_{\mathcal{Z},12}{}^2:=\alpha,\\[0.05in]
s_{\mathcal{Z},22}{}^2:=0,& s_{\mathcal{Z},22}{}^1:=1,&
s_{\mathcal{Z},21}{}^2:=\alpha,&s_{\mathcal{Z},21}{}^1:=\bar\alpha.
\end{array}$$

\subsection{The proof of Theorem~\ref{T1.5}}
Let $\lambda:=e^{2\pi\sqrt{-1}/3}$. Define an action of $\mathbb{Z}_3$ by setting
$T_\lambda f_1:=\lambda f_1$ and $T_\lambda f_2:=\bar\lambda f_2$. Then
$$\begin{array}{llll}
T_\lambda\tilde\Gamma_{11}{}^1=\lambda\tilde\Gamma_{11}{}^1,&
T_\lambda\tilde\Gamma_{11}{}^2=\tilde\Gamma_{11}{}^2,&
T_\lambda\tilde\Gamma_{12}{}^1=\bar\lambda\tilde\Gamma_{12}{}^1,&
T_\lambda\tilde\Gamma_{12}{}^2=\lambda\tilde\Gamma_{12}{}^2,\\
T_\lambda\tilde\Gamma_{22}{}^2=\bar\lambda\tilde\Gamma_{22}{}^2,&
T_\lambda\tilde\Gamma_{22}{}^1=\tilde\Gamma_{22}{}^1,&
T_\lambda\tilde\Gamma_{21}{}^2=\lambda\tilde\Gamma_{21}{}^2,&
T_\lambda\tilde\Gamma_{21}{}^1=\bar\lambda\tilde\Gamma_{21}{}^1.
\end{array}$$
The slices are equivariant with respect to this action, i.e. 
$$
T_\lambda s_{\mathcal{W}}(\alpha_1,\alpha_2)=s_{\mathcal{W}}(\lambda\alpha_1,\lambda\alpha_2)\text{ and }
T_\lambda s_{\mathcal{Z}}(\alpha)=s_{\mathcal{Z}}(\lambda\alpha)\,.
$$
The slices projects down to define local
coordinates on the oriented orbifolds where we must identify by the action of $\mathbb{Z}^3$ on $\mathbb{C}$
when dealing with $\mathfrak{Z}^+(2,0)$ and by the diagonal action of $\mathbb{Z}^3$ on $\mathbb{C}^2$
when dealing with $\mathfrak{W}^+(2,0)$.
This establishes Assertion~1 of Theorem~\ref{T1.5}.

Complex conjugation interchanges the roles of $f_1$ and $f_2$ and reverses the orientation. The slices are equivariant
with respect to the action of
complex conjugation
$$
\bar s_{\mathcal{W}}(\alpha_1,\alpha_2)=s_{\mathcal{W}}(\bar\alpha_1,\bar\alpha_2)
\text{ and }\bar s_{\mathcal{Z}}(\alpha)=s_{\mathcal{Z}}(\bar\alpha)\,.
$$
We adopt the notation of Definition~\ref{D1.4}
and let $\mathbb{Z}_3$ and complex conjugation generate the group $s_3$ which acts on $\mathbb{C}$ and on
$\mathbb{C}^2$.  The analysis of the orbifold structure performed in the orientable setting
now extends to the non-orientable setting; the role that $\mathbb{Z}_3$ plays as the orbifold group in the orientable setting is
now played by $s_3$ in the non-orientable setting.
\hfill\qed
\subsection{The boundary of $\mathfrak{Z}(2,0)$}\label{S4.4}
 We remark that the action of $s_3$ on $\mathbb{C}$ gives rise
to a corner with an angle of $\frac{2\pi}3$ at $[\Gamma_0]\in\mathfrak{Z}(2,0)$. Furthermore, the orbifold singularity
at $[\Gamma_0]\in\mathfrak{Z}^+(2,0)$ can be eliminated by using coordinates $z\rightarrow z^3$; $\mathfrak{Z}^+(2,0)$
has a smooth structure. Since the boundary of $\mathfrak{Z}(1,1)$ and $\mathfrak{Z}(0,2)$ is in fact smooth, Figure 1.1 is misleading
in this regard; the apparent cusp is a function of the parametrization using $(\psi_3,\Psi_3)$ and does not reflect the
underlying topological structure.

\section*{Acknowledgments} It is a pleasure to acknowledge useful conversations on this subject with
Professors M. Brozos-V\'{a}zquez, E. Garc\'{i}a-R\'{i}o, O. Kowalski, and J. H. Park. 
Research partially supported by project GRC2013-045 (Spain).


\begin{thebibliography}{lll}

\bibitem{A06}I. Agricola, ``The Srni lectures on non-integrable geometries with torsion",
{\it Archives Mathematicum (Brno) \bf 42} (2006) Supplement, 5--84.

\bibitem{ACF05} I. Agricola, S. Chiossi, and A. Fino, ``Solvmanifolds with integrable and non-integrable
$G_2$ structures", {\it J. Diff. Geo. and Appl. \bf 25} (2007), 125--135.

\bibitem{ACFH15} I. Agricola, S. Chiossi, T. Friedrich, J. H\"oll,
``Spinorial description of $SU(3)$ and $G_2$ manifolds", {\it J. Geom. Phys. \bf 98} (2015), 535-555.

\bibitem{AMK08} T. Arias-Marco and O. Kowalski,
``Classification of locally homogeneous affine connections with arbitrary torsion on 2-manifolds",
{\it Monatsh. Math. \bf 153} (2008), 1--18.

\bibitem{BGGP16} M. Brozos-V\'{a}zquez, \, E. Garc\'{i}a-R\'{i}o, and P. Gilkey,
``Homogeneous affine surfaces: Killing vector fields and gradient Ricci solitons",
http://arxiv.org/abs/1512.05515.

\bibitem{BGGP16a} M. Brozos-V\'{a}zquez, \, E. Garc\'{i}a-R\'{i}o, and P. Gilkey,
``Homogeneous affine surfaces: Moduli spaces",
http://arxiv.org/abs/1604.06610.

\bibitem{B15} S. Bunk, ``A method of deforming $G$-structures", {\it J. Geom. and Phys. \bf 15}, 72--80.

\bibitem{CGGV09}
E. Calvi\~no-Louzao, E. Garc\'{i}a-R\'{i}o, P. Gilkey, and R. V\'azquez-Lorenzo, 
``The geometry of modified Riemannian extensions", 
\emph{Proc. R. Soc. Lond. Ser. A Math. Phys. Eng. Sci.} \textbf{465} (2009), 2023--2040.

\bibitem{CGV10} 
E. Calvi\~{n}o-Louzao, E. Garc\'{i}a-R\'{i}o, and R. V\'{a}zquez-Lorenzo, 
``Riemann Extensions of Torsion-Free Connections with Degenerate Ricci Tensor",
{\it Canad. J. Math. \bf 62} (2010), 1037--1057.

\bibitem{De}
A. Derdzinski, 
``Noncompactness and maximum mobility of type III Ricci-flat self-dual neutral Walker four-manifolds", 
\emph{Q. J. Math.} \textbf{62} (2011), 363--395.


\bibitem{Du}
S. Dumitrescu, 
``Locally homogeneous rigid geometric structures on surfaces" 
\emph{Geom. Dedicata} \textbf{160} (2012), 71--90.

\bibitem{F15} A. Fino and L. Vezzoni, ``Special Hermitian metrics on compact solv manifolds", {\it J. Geom. and Phys. \bf 91}
(2015), 40--53.

\bibitem{FI02} Th. Friedrich and S. Ivanov, ``Parallel spinors and connections with skew-symmetric torsion in string theory",
{\it Asian J. Math. \bf6} (2002), 303--336.

\bibitem{FI03} Th. Friedrich and S. Ivanov, ``Almost contact manifolds, connections with torsion, and parallel spinors",
{\it J. Reine Angew. Math. \bf 559} (2003), 217--236.

\bibitem{GMW04} J. Gauntlett, D. Martelli, and D. Waldram, ``Superstrings with intrinsic torsion",
{\it Phy. Rev. D \bf 69} (2004), 086002.

\bibitem{G-SG}
A. Guillot and A. S\'anchez-Godinez,
``A classification of locally homogeneous affine connections on compact surfaces",
\emph{Ann. Global Anal. Geom.} \textbf{46} (2014), 335--349.

\bibitem{I04} S. Ivanov, ``Connections with torsion, parallel spinors, and geometry of spin(7) manifolds",
{\it Math. Research Letters \bf 11} (2004), 171--186.


\bibitem{IM14} M. Ivanova and M. Manev, ``A classification of the torsion tensors on almost contact manifolds
with B-metric", {\it Cent. Eur. J. Math. \bf 12} (2004), 1416--1432.

\bibitem{K10} M. Kassuba, ``Eigenvalue estimates for Dirac operators in geometries with torsion", {\it Ann. Glob. Anal. Geom. \bf 37}
(2010), 33--71.

\bibitem{KoSe}
O. Kowalski and M. Sekizawa, 
``The Riemann extensions with cyclic parallel Ricci tensor". 
\emph{Math. Nachr.} \textbf{287} (2014), 955--961.

 \bibitem{KVOp2}
O. Kowalski, B. Opozda, and Z. Vlasek, 
``A classification of locally homogeneous affine connections with skew-symmetric Ricci tensor on $2$-dimensional manifolds'', 
\emph{Monatsh. Math.} \textbf{130} (2000), 109--125.

\bibitem{KVOp}
O. Kowalski, B. Opozda, and Z. Vlasek,
``On locally nonhomogeneous pseudo-Riemannian manifolds with locally homogeneous Levi-Civita connections''. 
\emph{Internat. J. Math.} \textbf{14} (2003), 559--572.

\bibitem{KV03} O. Kowalski and Z. Vlasek, 
``On the local moduli space of locally
homogeneous affine connections in plane domains",
\emph{Comment. Math. Univ. Carolinae} \textbf{44} (2003), 229--234.

\bibitem{M11} M. Manev, ``A connection with parallel torsion on almost hypercomplex manifolds
with Hermitian and anti-Hermitian metrics", {\it J. Geom. Phys. \bf 61} (2011), 248--259.

\bibitem{M12} M. Manev, ``Natural connection with totally skew-symmetric torsion on almost contact manifolds with B-metric",
{\it Int. J. Geom. Methods. Mod. Phys. \bf 9} (2012), 125044.

\bibitem{Op04} B. Opozda, ``A classification of locally homogeneous connections on 2-dimensional manifolds",
J. Diff. Geo. Appl. {\bf 21} (2004), 173--198.

\bibitem{Opozda}
B. Opozda, ``Locally homogeneous affine connections on compact surfaces",
\emph{Proc. Amer. Math. Soc.} \textbf{132} (2004), 2713--2721.

\bibitem{S12} C. Stadtm\"uller, ``Adapted connections on metric contact manifolds", {\it J. Geom. Phys. \bf62} (2012), 2170--2187.

\bibitem{WWY14}
J. Wang, Y. Wang, and C. Yang, ``Dirac operators with torsion and the noncommutative residue for manifolds with boundary",
{\it J. Geom. and Phys. \bf 81} (2014), 92--111.
\end{thebibliography}
\end{document}